\documentclass{article}
\usepackage{amscd,amsfonts,euscript,amsmath,amsxtra,amssymb,amsthm}
\usepackage[english]{babel}
\newtheorem{theorem}{Theorem}

\newtheorem{propos}{Proposition}
\newtheorem{claim}{Claim}

\theoremstyle{definition} \newtheorem{example}{Example}
\newtheorem{remark}{Remark}
\newtheorem{convention}{Convention}

\theoremstyle{remark}

\textwidth13.5cm \numberwithin{equation}{section}

\usepackage[all]{xy}

\def\Q{{{\mathbb Q}}}

\def\P{{{\mathbb P }}}

\def\N{{{\mathbb N }}}

\def\FF{{{\mathcal F}}}

\def\LL{{{\mathcal L}}}

\def\OO{{{\mathcal O}}}

\def\Tor{{{\rm Tor }}}
\def\Spec{{{\rm Spec \,}}}

\def\Proj{{{\rm Proj \,}}}

\def\Supp{{{\rm Supp \,}}}

\def\dim{{{\rm dim \,}}}

\def\ker{{{\rm ker \,}}}

\def\length{{{\rm length \,}}}

\begin{document}
{\sl MSC 14A15, 14A05, 14B10, 14B25, 16D40}

 {\sl UDC 512.7}
\medskip

\begin{center}
{\Large\sc Infinitesimal Criterion for Flatness of Projective
Morphism of Schemes}\end{center}
\medskip
\begin{center}
Nadezda V. TIMOFEEVA

\smallskip

Yaroslavl' State University

Sovetskaya str. 14, 150000 Yaroslavl', Russia

e-mail: {\it ntimofeeva@list.ru}
\end{center}
\bigskip

\begin{quote}The generalization of the well-known criterion
for flatness of a projective morphism of Noetherian schemes
involving Hilbert polynomial, is given for the case of nonreduced
base of the morphism.

Bibliography: 4 titles.

\end{quote}

\footnotetext{The work was partially supported by the Institute of
Mathematics "Simion Stoilow"\, of Romanian Academy (IMAR)
(partnership IMAR -- BITDEFENDER) during of author's stay as
invited professor, June -- July 2011.}

\section*{Introduction} We start with some classical notation.
Let $\P^N_T$ be relative projective space of dimension $N$ over a
scheme $T$, $\OO(1)$ be a  line bundle on $\P^N_T$ generated by
hyperplane section. It is very ample relative to $T$. Also if
$f:X\to T$ is a morphism of schemes and $t\in T$ a closed point
with residue field $k(t)$ then $X_t:=f^{-1}(t)$ is a closed fibre
of $f$ over $t$.

The purpose of the present note is to generalize the following
well-known criterion for flatness of a projective morphism of
Noetherian schemes \cite[ch. III, theorem 9.9]{Hart}:
\begin{theorem}\label{classic} Let $T$ be an integral Noetherian scheme
and $X \subset \P_T^N$ be some closed subscheme. For each closed
point $t\in T$ take Hilbert polynomial $P_t \in \Q[m]$ of the
fibre $X_t$. This fibre is considered as closed subscheme in
$\P^N_t$. Then the subscheme  $X$ is flat over $T$ if and only if
Hilbert polynomial $P_t$ does not depend on the choice of $t$.
\end{theorem}

This theorem is applicable to any projective morphism of schemes
$f:X\to T$ with integral base scheme $T$, if one reformulate it as
follows:
\smallskip

{\it Let projective morphism of Noetherian schemes  $f:X\to T$
with integral scheme  $T$ fits into the commutative diagram
\begin{equation}\label{triangle}
\xymatrix{X \ar@{^(->}[r]^i \ar[rd]_f& \P^N_T \ar[d]\\
&T}
\end{equation}
with $i$ being closed immersion. It is flat if and only if  for an
invertible  $\OO_X$-sheaf $\LL$ which is very ample relative to
$T$ and such that  $\LL=i^{\ast} \OO(1),$ for every closed point
 $t\in T$ Hilbert polynomial of the fibre
$P_t(m)= \chi(\LL^m|_{X_t})$ does not depend on the choice of
$t\in T$.}
\smallskip

The proof of this theorem as presented in \cite{Hart} allows to
deduce flatness of any coherent $\OO_X$-sheaf $\FF$ over an
integral scheme $T$, if Hilbert polynomial $\chi (\FF \otimes
\LL^m|_{X_t})$ of its restriction on the fibre $X_t$ over each
point  $t\in T$ does not depend on the choice of $t$.

Cited criterion is not applicable in the case when the scheme $T$
has nonreduced scheme structure (\S \ref{example}). As we will
show below (\S \ref{proof}), this inconvenience can be removed if
Hilbert polynomial is replaced by some other function. In closed
points this function coincides with Hilbert polynomial of fibres.
We need some notation. If $t\in T$ is a closed point of the scheme
$T$ and this point corresponds to a sheaf of maximal ideals
${\mathfrak m}_t\subset \OO_X$, then the symbol $t^{(n)}$ stands
for $n$th infinitesimal neighborhood of the point $t\in T$. The
$n$th infinitesimal neighborhood is a subscheme defined by the
sheaf of ideals ${\mathfrak m}_t^{n+1}$ in $T$. In our
consideration $T$ is supposed to be Noetherian scheme of finite
type over a field, then for each $n\in \N$ the subscheme $t^{(n)}$
is zero-dimensional subscheme of finite length equal to
 $\length t^{(n)}=
\chi (\OO_{t^{(n)}})$. It is clear that this is positive integer
depending on both $t$ and $n$. If the point  $t=\Supp t^{(n)}$ is
known and fixed, we denote the length of $n$th infinitesimal
neighborhood $t^{(n)}$ by the symbol  $(n+1)$ (according to the
power of maximal ideal corresponding to the subscheme $t^{(n)}$).

We operate in the category of Noetherian schemes over a field $k$.
This field is supposed to be algebraically closed. The hypothesis
of algebraic closedness of the base field is essential in those
part of argument where we use filtrations (and cofiltrations) of
Artinian algebras. Namely, these are  proofs of claim \ref{claim1}
and of proposition \ref{propinf}. If $A/I$ is Artinian algebra
over an algebraically closed field, then $\length A/I =\dim_k
A/I$. Since all vector spaces appearing in this paper are defined
over the field $k$, the lower index in the notation of dimension
is omitted.

 To deduce that the morphism $f$ is flat, one has to examine preimages
$f^{-1}(t^{(n)}):=X\times _T t^{(n)}$ of infinitesimal
neighborhoods of reduced points $t\in T$. This provides data on
behavior of the morphism  $f$ over nonreduced scheme structure of
$T$. Since $T$ is of finite type, the power $n$ to be examined for
the given morphism  $f$, is bounded from above (and not greater
then maximal index of nilpotent elements in $\OO_T$ minus 1). This
paper is devoted to the proof of following results (theorem
\ref{critf} is a particular case of theorem \ref{critF}, and we
prove theorem \ref{critF} immediately).
\begin{theorem}\label{critf} Let a projective morphism of Noetherian schemes
of finite type $f: X \to T$ fits into commutative diagram
(\ref{triangle}). It is flat if and only if for an invertible
$\OO_X$-sheaf $\LL$ very ample relatively $T$ and such that
$\LL=i^{\ast} \OO(1)$, for any closed point  $t\in T$ the function
$$
\varpi_t^{(n)}(\OO_X,m)=\frac{\chi
(\LL^m|_{f^{-1}(t^{(n)})})}{\chi(\OO_{t^{(n)}})}
$$
does not depend on the choice of $t\in T$ and of $n\in \N$.
\end{theorem}
\begin{theorem}\label{critF} Let a projective morphism of Noetherian schemes
of finite type $f: X \to T$ fits into commutative diagram
(\ref{triangle}). The coherent sheaf of $\OO_X$-modules $\FF$ is
flat with respect to  $f$ (i.e. flat as $\OO_T$-module) if and
only if for an invertible $\OO_X$-sheaf $\LL$ very ample
relatively $T$ and such that $\LL= i^{\ast} \OO(1),$ for any
closed point $t\in T$ the function
$$
\varpi_t^{(n)}(\FF, m)=\frac{\chi (\FF \otimes
\LL^m|_{f^{-1}(t^{(n)})})}{\chi(\OO_{t^{(n)}})}
$$
does not depend on the choice of  $t\in T$ and of $n\in \N$.
\end{theorem}

In case when  $f$ is finite morphism the function in theorem
\ref{critf} takes the form
$$ \varpi_t^{(n)}(\OO_X,m)=\frac{\length
(f^{-1}(t^{(n)}))}{\length (t^{(n)})}.
$$

If the scheme $T$ is reduced, it is enough to examine only the
case  $n=0$. This corresponds to the classical situation
$\varpi_t^{(0)}(\OO_X,m)=P_t(m)$ (theorem \ref{classic}).

\section{Motivation}\label{example}
\begin{example} Consider nonreduced scheme $T=\Spec k[x]/(x^2)$ of length 2
 and a morphism  $f: X\to T$ if immersion of
(unique) closed point $X=\Spec k$. Since both schemes are
supported at a point, we replace examining of a morphism of
structure sheaves  $f^{\sharp}: \OO_T \to f_{\ast} \OO_X$ by the
study of a homomorphism $f^{\sharp}: k[x]/(x^2) \to k$ of
corresponding Artinian algebras. It is clear that $f^{\sharp}$ is
an epimorphism onto the quotient ring over nilradical:
$f^{\sharp}: k[x]/(x^2) \to (k[x]/(x^2))/Nil=k$. We use the
criterion for flatness of a ring homomorphism formulated in
\cite[ch. 1, proposition 2.1]{Milne}
\begin{propos} A homomorphism  $f: A\to
B$ is flat if and only if mappings $(a\otimes b \mapsto f(a)b): I
\otimes_A B \to B$ are injective for all ideals $I$ of $A$.
\end{propos}

Then it is necessary to test  a homomorphism $(x)
\otimes_{k[x]/(x^2)} k \to k$, for injectivity. This homomorphism
is induced by the inclusion of the ideal $(x) \hookrightarrow
k[x]/(x^2)$. The tensor product  $(x) \otimes_{k[x]/(x^2)} k$ is
nonzero and is a $k$-linear span of the element $x\otimes 1$,
$x^2=0$. The element $x\otimes 1$ is taken to 0 under the mapping
to $k$. Then the ring homomorphism of interest and the
corresponding scheme morphism are not flat.

The same result is given by theorem теоремой \ref{critf}. The
Hilbert polynomial of the fibre of the morphism $f$ over the
unique closed point  $t$ of the scheme  $T$ equals to $P_t(m)=1$.
The function  $\varpi$ if computed for 1st infinitesimal
neighborhood of the closed point on the base $T$ (it coincides
with the whole of the scheme $T$), equals to  $\varpi^{(1)}_t(m)
=1/2$. This differs from the value  $\varpi^{(0)}_t (m)=P_t(m)=1.$
\end{example}
\begin{example} Let $A= k[x]/(x^3)$, and $B=k[x]/(x^2)$ is
$A$-module of interest. It is clear that $B$ is finitely generated
(and has one generator) over $A$. $A$ is local $k$-algebra with
maximal ideal $(x)$ and residue field $k$. Since  $B$ is not free
as $A$-module that $B$ is not flat as  $A$-module. On the other
hand, the group $\Tor ^A_1(k,B)$ fits into the exact sequence
which is induced by tensoring of a triple
$$ 0\to (x^2) \to A \to B \to 0
$$ by $\otimes _A k$:
\begin{equation}0\to \Tor^A_1(k,B) \to (x^2)\otimes_A k \to
A\otimes _A k\to B\otimes _A k \to 0
\end{equation}
This sequence implies that $\Tor ^A_1(k,B)=(x^2) \otimes _A k= k$.
This also shows that  $B$ is not flat as $A$-module. Computing the
function $\varpi$ one has  $\varpi^{(0)}_t (m)=\varpi^{(1)}_t(m)
=1$, $\varpi^{(2)}_t (m)=2/3$.
\end{example}

\section{Algebraic version}\label{proof}
We will need the following criterion for flatness \cite[ch. 1,
theorem 7.8]{Mats}.
\begin{propos}\label{torcrit} $A$-module $M$ is flat if and only if
$\Tor ^A_1(A/I, M)=0$ for any finitely generated ideal $I\subset
A$.
\end{propos}
\begin{convention} Let $A$ be local Noetherian  $k$-algebra
with residue field $k=\overline k$, $I\subset A$ be an ideal such
that  $A/I$ is Artinian $k$-algebra of length  $n$, i.e. $\dim
A/I=n$. Then the ideal $I$ is said to be of colength $n$. This
fact  will be reflected in the notation of the ideal: we write
$I_n$ instead of $I$.
\end{convention}





 \begin{propos}\label{prcrmod} Let  $M$ be a finitely generated
module over the local Noetherian  $k$-algebra  $A$ with residue
field $k$. Module  $M$ is free if and only if for all $n>0$ and
for all ideals  $I_n\subset A$ of colength $n$
\begin{equation}\label{critmod}\frac{\dim (M\otimes_A A/I_n)}{n} =\dim (M\otimes_A
k).\end{equation}
\end{propos}
\begin{proof} Note that  $M/I_n M=M\otimes_A A/I_n $ and
$M/I_nM \otimes_{A/I_n} k=M\otimes _A A/I_n \otimes
_{A/I_n}k=M\otimes _A k.$ Analogously,  $$M/I_nM
\otimes_{A/I_n}A/I_{n-1}= M \otimes_A A/I_{n-1}=M/I_{n-1}M.$$
There is an exact triple of $A$-modules  (and of $A/I_n$-modules)
\begin{equation}\label{triple} 0\to {\mathfrak m}_n \to A/I_n \to k \to 0.
\end{equation}
For the maximal ideal  ${\mathfrak m}_n\subset A/I_n$ we have
$M\otimes_A{\mathfrak m}_n=M \otimes_A A/I_n \otimes _{A/I_n}
{\mathfrak m}_n=M/I_n M \otimes _{A/I_n} {\mathfrak m}_n$.
Tensoring of  (\ref{triple}) by $M/I_n M \otimes_{A/I_n}$ yields
in the exact sequence
\begin{equation}\label{tortail}0\to \Tor_1^{A/I_n}(M/I_nM,k) \to
M/I_n M \otimes_{A/I_n} {\mathfrak m}_n \to M/I_nM \to M/I_n M
\otimes_{A/I_n} k \to 0.\end{equation} Left exactness is
guaranteed by $\Tor_1^{A/I_n}(M/I_nM, A/I_n)=0$ because any ring
is flat over itself.
\begin{claim}\label{claim1} The equality  (\ref{critmod}) implies
exactness of a sequence
\begin{equation}\label{trin}0\to M/I_n M \otimes_{A/I_n}
{\mathfrak m}_n \to M/I_nM \to M/I_n M \otimes_{A/I_n} k \to
0.\end{equation}
\end{claim}
The proof of this claim will be presented below, when the proof of
the proposition will be completed. Claim \ref{claim1} and exact
sequence  (\ref{tortail}) imply that $\Tor_1^{A/I_n}(M/I_nM,k)=0.$
By  proposition \ref{torcrit} $M/I_nM$ is flat as  $A/I_n$-module.
Now we can consider not all possible ideals of finite colength but
powers of maximal ideal ${\mathfrak m} \subset A$ only. Passing to
${\mathfrak m}$-adic completions  $\widehat A$ and $\widehat M$ of
the ring $A$ and of the module $M$ respectively, we get
\cite[proof of theorem 22.4(ii)]{Mats} that  $\widehat M$ is flat
$\widehat A$-module.

By the same theorem \cite[theorem 22.4(ii)]{Mats} we conclude that
 $A$-module $M$ is flat.

The proof of the opposite implication is trivial. If the finitely
generated module over the local ring is flat then it is free, i.e.
$M \cong A^q$. This implies equalities of the form (\ref{critmod})
for all $n>0$ and for all ideals $I_n\subset A$. This completes
the proof of the proposition.
 \end{proof}

Now we prove claim \ref{claim1}. To organize induction over
 $n$ consider exact diagrams of the form
 \begin{equation}\label{descent} \xymatrix{&&0&0&\\
 &&k\ar[u] \ar@{=}[r]& k \ar[u]\\
 0\ar[r]&k\ar[r]& A/I_n \ar[u] \ar[r]& A/I_{n-1} \ar[u] \ar[r]&0\\
 0\ar[r]&k\ar@{=}[u] \ar[r]& {\mathfrak m}_n \ar[u] \ar[r]&
 {\mathfrak m}_{n-1} \ar[u] \ar[r]& 0\\
 && 0\ar[u]& 0\ar[u]}
 \end{equation}
Such a diagram of  $A$-modules  (and of  $k$-algebras) can be
built up for any $n>0$ and for any ideal  $I_n \subset A$ of
colength $n$. For $n=2$ we have ${\mathfrak m}_n={\mathfrak m}_2
\cong k$,
 ${\mathfrak m}_{n-1}={\mathfrak m}_1=0.$

Let  (\ref{critmod}) holds. Then  $$\dim M\otimes_A A/I_n=\dim
M\otimes _A A/I_{n-1}+ \dim M \otimes _A
 k.$$ This implies that the triple  $0\to M\otimes_A k \to M\otimes_A A/I_n \to M\otimes _A A/I_{n-1} \to
 0$ is exact.
 \begin{remark} Exactness of this triple  a priori
 does not imply that $$\Tor_1^A(M, A/I_{n-1})=\Tor_1^A(M, A/I_n)=\Tor_1^A(M,k)=0.$$
 This result follows from proposition
 \ref{prcrmod}.
 \end{remark}
Tensoring of the diagram (\ref{descent}) by $M \otimes_A$ leads to
the diagram
\begin{equation}\label{descM}\xymatrix{&&0&0\\
&&M\otimes_A k \ar[u] \ar@{=}[r]& M\otimes_A k \ar[u]\\
0 \ar[r]& M\otimes _A k \ar[r]& M\otimes_A A/I_n \ar[u] \ar[r]&
M\otimes_A A/I_{n-1} \ar[u] \ar [r]&0\\
& M\otimes_A k \ar@{=}[u] \ar[r]& M \otimes_A {\mathfrak m}_n
\ar[u] \ar[r]& M\otimes_A {\mathfrak m}_{n-1} \ar[u] \ar[r]&0}
\end{equation}
Let  $R:= \ker (M\otimes_A {\mathfrak m}_n \to M\otimes_A
{\mathfrak m}_{n-1}).$ Then the isomorphism $M\otimes_A k \cong
M\otimes_A k$ factors as  $M\otimes_A k \twoheadrightarrow R \to
M\otimes_A k$. This implies that  $R \cong M\otimes_A k,$ and
lower horizontal triple in (\ref{descM}) is left-exact.
Consequently, $\dim M\otimes_A {\mathfrak m}_n = \dim M \otimes_A
{\mathfrak m}_{n-1}+ \dim M\otimes_A k$. Applying induction over
 $n$ we have  $\dim M \otimes_A {\mathfrak m}_n=
(n-1)\dim M\otimes_A k$. This implies that the triple  $M\otimes
_A {\mathfrak m}_n \to M\otimes _A A/I_n \to M\otimes_A k \to 0$
(which is equivalent to the triple (\ref{trin})) is left-exact.
This proves the claim.

\begin{propos} \label{propinf} Equalities  (\ref{critmod}) hold
for all  $n>0$ and for all  $I_n \subset A$ if and only if the
analogous equalities
\begin{equation}\label{critmod1} \dim M\otimes_A A/{\mathfrak m}^n
= \dim A/{\mathfrak m}^n \,\dim M\otimes_A k,
\end{equation} hold for ${\mathfrak m}^n$ for all $n>0$.
\end{propos}
\begin{proof} Part "only if"\, is obvious; it rests to prove part "if".
Denote by the symbol  $(n)$ the length of the quotient algebra
$A/{\mathfrak m}^n$, i.e. $(n):=\dim A/{\mathfrak m}^n$. For
descending induction on lengths of quotient algebras we write down
exact sequences of the form
\begin{eqnarray} 0\to k \to A/{\mathfrak m}^n
\to A/I_{(n)-1}\to 0, \nonumber \\
0\to k \to A/I_{(n)-1} \to A/I_{(n)-2} \to 0, \nonumber\\
\dots \dots \dots \dots \dots \dots \dots \dots \dots \dots \dots \dots \nonumber\\
0\to k \to A/I_{(n')+1} \to A/{\mathfrak m}^{n'} \to 0\nonumber
\end{eqnarray}
for appropriate  $n'<n$. Tensoring by $M\otimes_A$ and dimension
counting lead to the sequence of inequalities
\begin{eqnarray}\label{ineq} \dim M\otimes_A A/{\mathfrak m}^n \le \dim M
\otimes_A k + \dim M \otimes_A
A/I_{(n)-1}, \\
\dim M \otimes_A A/I_{(n)-1}\le \dim M \otimes_A k+\dim M
\otimes_A A/I_{(n)-2},\nonumber\\
\dots \dots \dots \dots \dots \dots \dots \dots \dots \dots
\dots \dots \dots \dots \dots \dots \dots \dots \nonumber\\
\dim A/I_{(n')+1}\le \dim M \otimes_A k+\dim A/{\mathfrak m}^{n'}.
\nonumber
\end{eqnarray}
Continuing descent till  $n'=1$ and applying (\ref{critmod1}) we
conclude that inequalities in  (\ref{ineq}) are indeed equalities.

Then for any $I_l, l>0$ there exist  $n>0$ such that $A/{\mathfrak
m}^n \twoheadrightarrow A/I_l$. In this case there is a
cofiltration $A/{\mathfrak m}^n \twoheadrightarrow
A/I_{(n)-1}\twoheadrightarrow \dots \twoheadrightarrow A/I_l
\twoheadrightarrow \dots \twoheadrightarrow k\twoheadrightarrow 0$
of length $(n)$ with kernels isomorphic to  $k$, and such that it
contains  $A/I_l$. Counting of dimensions of vector spaces
 $M\otimes_A A/I_j$, $j=1, \dots , l,$ and application of equalities
  (\ref{ineq}) yield in the required equality
$\dim M\otimes _A A/I_l =l \dim M\otimes_A k$.
\end{proof}

\section{Proof for coherent  $\OO_T$-module}

Since the assertion of the theorem is local in $T$ one can assume
that  $T=\Spec A$ for local Noetherian $k$-algebra  $A$. In
further text we will use the notation $\FF(m):= \FF
\otimes_{\OO_X} \LL^m$. The group  $H^0(\Spec A, f_{\ast} \FF(m))=
H^0(X, \FF(m))$ carries a structure of finitely generated
$A$-module. It is necessary to prove that this module is flat. For
this purpose consider a finite presentation of the quotient ring
$A/I_n$ (it exists because the ideal $I_n$ is finitely generated):
\begin{equation}\label{pres}
A^q \to A \to A/I_n \to 0. \end{equation} The quotient ring
 $A/I_n$ fixes a zero-dimensional subscheme  $Z\subset T$ of
 length  $n$. The presentation (\ref{pres}) induces the triple
 $$
\FF(m)^q \to \FF(m) \to \FF(m)\otimes_{\OO_X}f^{\ast} \OO_Z \to 0.
$$
Formation of groups of global sections leads to sequences
\begin{equation}\label{pres1}H^0(X,\FF(m))^q \to
H^0(X,\FF(m)) \to H^0(X,\FF(m)\otimes_{\OO_X}f^{\ast} \OO_Z )\to
0.
\end{equation}
Right-exactness is achieved here when $m\gg 0$. Tensoring of
presentation (\ref{pres}) by $H^0(X, \FF(m))\otimes_A$ leads to
the right-exact triple
\begin{equation}\label{pres2} H^0(X, \FF(m))^q \to H^0(X, \FF(m))
\to H^0(X, \FF(m))\otimes_A A/I_n \to 0.
\end{equation}
Comparison of  (\ref{pres1}) and (\ref{pres2}) yields in the
isomorphism
\begin{equation}\label{sect} H^0(X,\FF(m)\otimes_{\OO_X}f^{\ast} \OO_Z )=
H^0(f^{-1}Z,\FF(m)\otimes_{\OO_X}f^{\ast} \OO_Z )\cong H^0(X,
\FF(m))\otimes_A A/I_n.\end{equation}

We suppose that  \begin{equation}\label{dim}\dim
H^0(f^{-1}Z,\FF(m)\otimes_{\OO_X}f^{\ast} \OO_Z )=n \dim
H^0(f^{-1}t, \FF(m) \otimes_{\OO_X} f^{\ast}k_t),\end{equation}
where  $t$ is the unique closed point of scheme $T$. By
(\ref{sect}) we have
\begin{eqnarray} H^0(f^{-1}Z, \FF(m) \otimes_{\OO_X} f^{\ast}
\OO_Z) &\cong& H^0(X, \FF(m)) \otimes_A A/I_n, \nonumber
\\ H^0(f^{-1 }t,\FF(m)\otimes_{\OO_X} f^{\ast} k_t)&\cong & H^0(X,
\FF(m))\otimes_A k. \nonumber \end{eqnarray}

Substituting these isomorphisms into (\ref{dim}) we conclude that
for all  $n>0$ and for all  $I_n \subset A$ the following
equalities hold:
$$ \dim H^0(X, \FF(m))\otimes_A A/I_n=
n \,\dim H^0(X, \FF(m))\otimes_A k,
$$
This validates proposition \ref{prcrmod} for  $H^0(X, \FF(m))$.
Hence $H^0(X, \FF(m))$ is flat $A$-module.

The proof of flatness of  $\FF$ as  $\OO_T$-module copies the
proof of the implication (ii)$\Rightarrow$(i) in \cite[ch. III,
proof of theorem 9.9]{Hart} verbatim. This part of the cited proof
remains valid also for nonreduced scheme $T$. By projectivity of
the morphism $f$ we can restrict to the case when $f$ is a
structure morphism of some projective bundle $f: \Proj A[x_0:
\dots : x_n] \to \Spec A$ and consider graded $A[x_0: \dots : X_n
]$-module $M= \bigoplus_{m\ge m_0} H^0(X, \FF(m))$. The integer
$m_0$ is chosen as big as  $A$-modules $H^0(X, \FF(m))$ are free
for all $m\ge m_0.$ Then  $\FF=\widetilde M$, where \,
$\widetilde{}$\, denotes formation of a coherent sheaf of
$\OO_{\Spec A}$-modules which is associated with finitely
generated  $A$-module $M$ ("sheafification"). In this case  $M$
and $\bigoplus_{m\ge 0} H^0(X, \FF(m))$ are equal for all $m\ge
m_0$ and hence \cite[ch. II, proposition 5.15]{Hart} $\widetilde M
= (\bigoplus_{m\ge 0} H^0(X, \FF(m)))^{\sim}$. Since $M$ is free
(and, consequently, flat) $A$-module, then $\FF$ is flat over $A$
(and hence over $T=\Spec A$.)

The proof of opposite implication repeats the proof of implication
(i)$\Rightarrow $(ii) in \cite[ch. III, proof of theorem
 9.9]{Hart} verbatim. Let $\FF$ be a flat
$\OO_T$-module and we reduce our consideration to the case
$X=\Proj A[x_0: \dots : X_n]$, $T=\Spec A$, for Noetherian local
ring  $A$. Compute  $H^i(X, \FF(m))$ as \v{C}ech cohomology of
standard open affine covering $\mathfrak U$ of the space $X$.
Namely, $H^i(X, \FF(m))=H^i(C^{\bullet} ({\mathfrak U}, \FF(m))).$
Since the sheaf $\FF$ is flat, then for all  $i\ge 0$ the term
$C^i ({\mathfrak U}, \FF(m))$ is flat $A$-module. If $i>0$ then
for $m\gg 0$ we have $H^i(X, \FF(m))=0.$ Then  \v{C}ech complex
provides a right resolution for $A$-module $H^0(X, \FF(m)),$  and
the sequence
$$
0\to H^0(X, \FF(m)) \to C^0({\mathfrak U}, \FF(m)) \to \dots \to
C^n({\mathfrak U}, \FF(m)) \to 0
$$
is exact. Since all terms of \v{C}ech complex are flat
$A$-modules, then cutting this exact sequence into exact triples
we come to flatness of $A$-module $H^0(X, \FF(m))$. Then it is
subject of proposition \ref{prcrmod}, and for all $n>0$ and for
all  $I_n \subset A$ following equalities hold:
$$ \dim H^0(X, \FF(m))\otimes_A A/I_n=
n \,\dim H^0(X, \FF(m))\otimes_A k.
$$
By the isomorphism  (\ref{sect}) which was proven independently,
assertions of theorem  \ref{critF} are fulfilled.

{\bf Acknowledgements.} Author expresses her deep and sincere
gratitude to Prof. Dr. Vasile Brinzanescu (IMAR, Bucharest,
Romania) for drawing the attention to the question. Also the
author thanks the Institute of Mathematics of Romanian Academy
(IMAR), where the part of this work was done, for hospitality and
support.

\end{document}